\documentclass[11pt]{article}

\author{
Patrick Dondl
\thanks{Abteilung f{\"u}r Angewandte Mathematik, 
Albert-Ludwigs-Universit{\"a}t Freiburg, 
Hermann-Herder-Str.~10, 
79104 Freiburg, Germany; 
{\tt patrick.dondl@mathematik.uni-freiburg.de} }, 
Martin Jesenko
\thanks{Abteilung f{\"u}r Angewandte Mathematik, 
Albert-Ludwigs-Universit{\"a}t Freiburg, 
Hermann-Herder-Str.~10, 
79104 Freiburg, Germany; 
{\tt martin.jesenko@mathematik.uni-freiburg.de} }, 
Michael Scheutzow
\thanks{
Institut f\"{u}r Mathematik, 
MA 7--5, 
Fakult\"{a}t II,
Technische Universit\"{a}t Berlin, 
Stra{\ss}e des 17.~Juni~136,
10623 Berlin, 
Germany;
{\tt ms@math.tu-berlin.de}
}
}

\title{Pinning of interfaces in a random medium with zero mean}
\usepackage[a4paper]{geometry}
\usepackage{enumerate}
\usepackage{mathrsfs}
\usepackage{pifont,stmaryrd}
\usepackage{fancyhdr}
\usepackage{amssymb, amsmath, amstext, amsthm}
\usepackage{titling}
\usepackage{eucal}
\usepackage{aurical}
\usepackage[T1]{fontenc}
\usepackage{pictexwd,dcpic}
\usepackage{graphicx}
\usepackage{wrapfig, float}
\usepackage{todonotes}

\newtheorem{theo}{Theorem}[section] 

\newtheorem{cor}[theo]{Corollary}
\newtheorem{prop}[theo]{Proposition}
\newtheorem{con}[theo]{Conclusion}

\theoremstyle{definition}

\newtheorem{remark}[theo]{Remark}

\def\E {\mathbb{E}}
\def\P {\mathbb{P}}
\def\Fcal {\mathcal{F}}

\def\Laplace {\Delta}

\def\R {\mathbb R}

\def\N {\mathbb N}
\def\Z {\mathbb Z}

\def\i {\infty}

\def\* {$*$}

\def\supp{\mathop{\rm supp}}
\def\dist{\mathop{\rm dist}}

\usepackage{color}

\newcommand{\cred}{\color{red}}

\begin{document}

\maketitle

\begin{abstract}
We consider two related models for the propagation of a curvature sensitive interface in a time independent random medium. In both cases we suppose that the medium contains obstacles that act on the propagation of the interface with an inhibitory or an acceleratory force. We show that the interface remains bounded for all times even when a small constant external driving force is applied. This phenomenon has already been known when only inhibitory obstacles are present. In this work we extend this result to the case of---for example---a random medium of random zero mean forcing.

The first model we study is discrete with a random forcing on each lattice site. In this case we construct a supersolution employing a local path optimization procedure. In the second, continuous, model we consider a random heterogenous medium consisting of localized small obstacles of random sign. To construct a stationary supersolution here, we need to pass through sufficiently many blocking obstacles while avoiding any obstacles of the other sign. This is done by employing a custom percolation argument.

\smallskip \noindent
{\bf 2020 Mathematics Subject Classification.} 35R60, 74N20.

\smallskip \noindent
{\bf Key words and phrases.} Quenched Edwards-Wilkinson model, phase boundaries, pinning, random environment, viscosity super-solutions

\end{abstract}

%

\section{Introduction}

We are investigating a model for an interface propagating through a random, heterogeneous medium. As a governing equation, we consider
\begin{equation} \label{eq:evolution}
 \partial_{t} u(t,x) = \Laplace u(t,x) - f(x,u(t,x)) + F. 
\end{equation}

This equation arises for example as a linearization of the mean curvature flow with a spatially non-homogeneous driving force (see, e.g., \cite{DirrYip,DDS} for a derivation). The graph of $u$ is the shape of the interface at time $t$, $F$ is a given exterior driving force and $f$ is the force exerted by the medium on the interface. The function $f$ is supposed to be non-zero only in small sets which correspond, for example, to obstacles, impurities or precipitates in the medium. Such a setting is commonly found when studying magnetic domain values, dislocation lines, or charge density waves. For a more detailed list of related problems and applications, we refer to \cite{BrazovskiiNattermann,DirrYip}.

In the deterministic, periodic setting this problem was studied in \cite{DirrYip}, with $f$ having zero mean (or at least this being a sufficient condition for their results). The authors show that in this case, under some non-degeneracy conditions, there exists a critical driving force $F^*$, up to which \emph{pinning} occurs, i.e., the evolving interface stays below some stationary hypersurface for all times. For larger driving forces, there is a non-stationary solution to the problem that propagates with positive average velocity.

Here, we consider a random setting, which can be seen as a variant of the quenched Edwards-Wilkinson model, in the sense that the strengths and positions of obstacles are random and non-correlated on long distances, but time-independent. Furthermore, they may act on the interface in both directions (i.e., $f$ is not assumed to be non-negative). Thus, our model also includes the case with random forces having zero mean, which is a common setting in the physics literature, see, e.g., \cite{BrazovskiiNattermann,Nattermann:97}.

Our goal is to show a result regarding pinning for two models.  In both models we consider only the one dimensional case, so that the interface is in fact a curve, and our equation reads
\[ \partial_{t} u(t,x,\omega) = \partial_{xx} u(t,x,\omega) - f(x,u(t,x,\omega), \omega ) + F \]
where exceptionally we stressed the randomness of the setting given by a probability space $ ( \Omega , {\mathcal{F}} , \mathbb{P} ) $ by writing the random variable $ \omega $.
The interface is initially supposed to be flat having the height 0, i.e.~$ u ( 0, \cdot ) = 0 $. The basic idea is to a.s.~find a (viscosity) supersolution (for the definition and properties see, e.g., \cite{Crandall:1992kn}) to the related stationary problem, i.e., a function $v$ that satisfies
\[ v''(x) - f(x,v(x)) + F \le 0
\quad \mbox{and} \quad
v(x) \ge 0
\quad \mbox{for all } x \in \R.  \] 
By employing an appropriate comparison principle, this immediately implies that the interface a.s. stays below the graph of $v$ for all times since this was the case at $ t = 0 $. Our main goal in this work is thus to show the existence of such a non-negative stationary supersolution in the setting of our two models.
 
The first model we study is spatially purely discrete; we there consider the lattice $\Z^{2}$. Each lattice point acts with a force of random strength chosen by a suitable probability distribution. The notions of the space and time derivative are adapted to the discrete case. This model was studied in, e.g., \cite{Bodineau:2013ur}, where for some very specific distributions of $f$ pinning and depinning results were shown. Here, we focus on results regarding pinning, but consider a large class of distributions (in particular, allowing the aforementioned case of $f$ having zero mean).

For the continuous setting, for the case $f\ge 0$, the occurrence of pinning for sufficiently small driving force was shown in \cite{DDS}. Depinning results for unbounded obstacles were studied, e.g., in \cite{DondlScheutzow:12, DondlScheutzow:17}. To prove pinning results in the case of obstacles without prescribed sign, we follow a similar strategy of localization and percolation as \cite{DDS}, but require a more explicit form of the constructed supersolution. This is possible due to our one-dimensional setting. Moreover, it is necessary to extend the percolation result from \cite{DDGHS} to the case with finite dependence of sites.

\section{Discrete model on $\Z^2$}

Let the obstacle strenghts $f(i,j)$, $i,j \in \Z$ be independent and identically distributed $\Z \cup \{-\infty\}$-valued random variables
defined on a probability space $(\Omega,\Fcal,\P)$. We denote the expected value of a (possibly extended real-valued)
random variable on that space by $\E$ (whenever it is defined). 
We consider the following continuous time evolution of (random) functions $u_t: \Z \to \Z$, $t \ge 0$.
The initial condition is $u_0\equiv 0$. The function $u$ can jump from its current value $u_t(i)$ to $u_t(i)+1$
or to  $u_t(i)-1$ depending on the value of $f(i,u_t(i))$ and the discrete Laplacian $\Delta_{1} u_t(i)$ defined as 
\[ \Delta_{1} u(i) = u(i+1) + u(i-1) - 2 u(i). \]
The corresponding  {\em jump rate} is $\lambda=\Lambda\big(\Delta_{1} u_t(i)-f(i,u_t(i))\big)$ with the interpretation
that if $\lambda >0$, then $u$ can only jump to $u_t(i)+1$ (with rate $\lambda$) and when $\lambda <0$,
then $u$ can only jump to $u_t(i)-1$ and does so with rate $-\lambda$. 
Here, $\Lambda$ is a strictly increasing function from $\Z$ to $\R$ which satisfies $\Lambda(0)=0$. The phrase
\emph{ $u$ jumps from $u_t(i)$ to $u_t(i)+1$ with rate $\lambda >0$} means - as usual - that for some 
exponentially distributed random variable $\xi$ with parameter $\lambda$ which is independent of the field  $f$, $u_s(i)=u_t(i)$
for all $s \le \zeta \wedge (t +\xi)$, where $\zeta$ is the first time after $t$ when the Laplacian at $u_t(i)$ changes
(due to a jump of one of the neighbors) and where $u_{t+\xi}(i)=u_t(i)+1$ in case $\xi<\zeta$.

One may ask under which conditions on $\Lambda$ as above there exists a unique process with values in
the functions from $\Z$ to $\Z$ associated to the given rates. This clearly holds when 
$\Lambda$ is bounded which we can safely assume since none of the following results in this section depends on
$\Lambda$ (except for its sign).



Let $Z, Z_0, Z_1,...$ be independent  random variables which have the same distribution as $f(i,j)$.

\begin{theo}
  If
  $$
  \E  \big(Z_0 \vee (-1+Z_1) \vee (-2+Z_2)\vee ...\big)>0,
$$
then, almost surely, there exists a function  $v :\Z \to \N_0$ such that $\Delta_{1} v(i) \le f(i,v(i))$, i.e., a
non-negative supersolution.
\end{theo}

\begin{proof} Without loss of generality we can and will assume that $Z$ is essentially bounded from above. 
  We first construct a supersolution $v$ and show that it is almost surely bounded from below. Then the claim will follow
  easily.

  We start by defining $v(0)$. Fix $N \in \N_0$, and let $M$ be the smallest integer for which $M \ge N$ and $f(0,M)>-\infty$.
  The condition in the theorem guarantees that such an $M$ exists.  Let $v(0):=M$.  We will successively construct
  $v(n+1)$ from the previous values $v(0),...,v(n)$ and the obstacle strengths $f(1,j),...,f(n+1,j)$, $j \in \Z$ (the construction for negative values is completely analogous).
  We will always first define a {\em provisional value}
   $\bar v(n+1)$ which we will then (possibly) change to the smaller final value $v(n+1)$ depending on the obstacles
   $f(n+1,j)$, $j \in \Z$. We will perform the construction in such a way that the sequence $v(0),...,v(n),\,\bar v(n+1)$ satisfies the condition of a supersolution at $1,...,n$.
   Note that when we later change $\bar v(n+1)$ to a smaller value $v(n+1)$ then this property still holds for the sequence  $v(0),...,v(n),\, v(n+1)$. 

   We define the provisional values $\bar v(1) =\bar v(-1)$ in such a way that the sequence $\bar v(-1)$, $v(0)$, $\bar v(1)$ satisfies the condition of a supersolution at 0.
   Specifically, we choose $\bar v(1)=\bar v(-1):=v(0)+\Big\lfloor \frac {f(0,v(0))}2\Big \rfloor$.

   Given $v(n)$ and $\bar v(n+1)$ for some $n \in \N_0$, we now define $v(n+1)$ and $\bar v(n+2)$ in such a way that the increment
   $\bar v(n+2)-v(n+1)$ is as large as possible subject to the conditions $v(n+1) \le \bar v(n+1)$ and the supersolution condition at $n+1$.
   The idea behind this choice is that large increments are likely to ensure that $v$ is bounded from below. 
   Clearly, the optimal choices are $v(n+1):=\bar v(n+1)-m$ where $m\in \N_0$ maximizes
   $f(n+1,\bar v(n+1)-m)-m$ and $\bar v(n+2):=2 v(n+1)-v(n)+f(n+1,v(n+1))$ (which satisfies the supersolution condition at $n+1$ with equality). Note that a maximizing $m$ exists
   since $f$ is essentially bounded from above.
   For an illustration of this procedure, see Figure~\ref{fig:discrete_path} below. 
      
   Note that the sequence $v(n)$, $n \in \Z$ is a supersolution by construction.  We check that the function $n\mapsto v(n)$ is almost surely bounded from below.

Define
$D_n:=\bar v(n)-v(n-1),\,n \ge 1$. By construction (and our independence assumptions) the random variables
$D_{n+1}-D_{n}$, $n \in \N$ are independent and identically distributed  with expectation 
$$
\alpha:=\E \big(D_{n+1}-D_n\big)=\E \big(Z_0 \vee (-1+Z_1) \vee (-2+Z_2)\vee ...\big)>0.
$$
(Note that independence of the increments is generally lost if we delete the bar in the definition of $D_n$.)
We have, for $n \in \N$,
$$
v(n)=v(n)-\bar v(n)+D_{n}+v(n-1).
$$
Note that the sequence $v(n)-\bar v(n)$, $n \ge 1$ is i.i.d.~with expected value larger than $-\infty$
which implies $\lim_{n \to \infty} (v(n)-\bar v(n))/n=0$ almost surely. Further, by the strong law of large numbers,  
$\lim_{n \to \infty}D_n/n=\alpha>0$ almost surely. Hence,
$$
\frac 1n\big(v(n)-v(n-1)\big)=\frac 1n\big(v(n)-\bar v(n)\big)+\frac 1n D_n
$$
converges to $\alpha>0$ almost surely. In particular, $v(n)\ge v(n-1)$ for all sufficiently large $n$
(and, analogously, $v(-n)\ge v(-n+1)$ for all sufficiently large $n$).
In particular, the function $v$ is almost surely bounded from below. Using translation invariance, we see that
  the probability that the function $v$ is non-negative converges to 1 as the initially chosen $N$ converges to $\infty$,
  i.e.~we have proven the almost sure existence of a non-negative supersolution.
\end{proof}

\begin{figure}[ht]
\begin{center}
\includegraphics[width=0.8\textwidth]{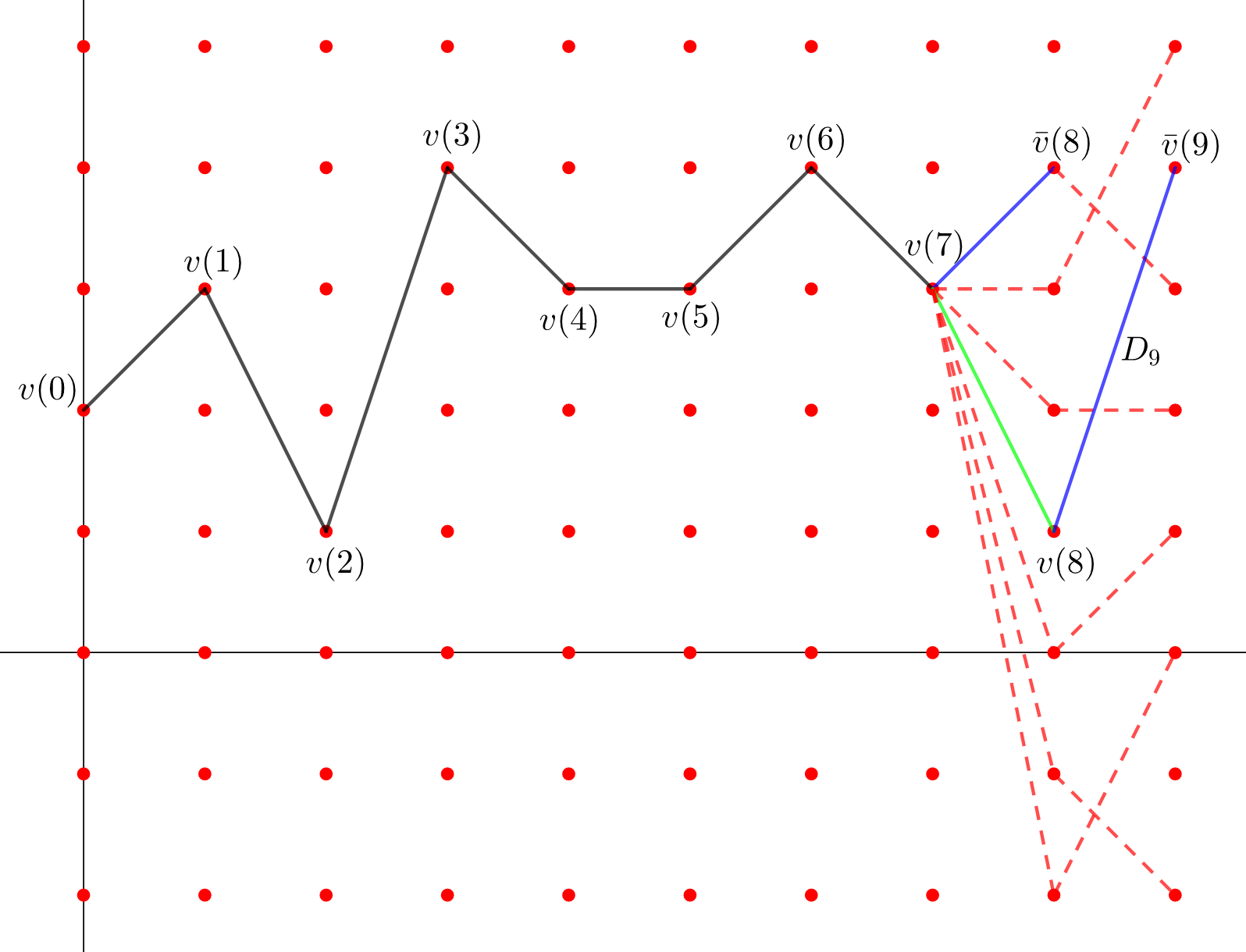}
\caption{Discrete path optimization. Suppose that $ v(0) , \ldots , v(7) $ are already set. 
Then $ \bar{v}(8) $ (connected by the blue line) is determined as the maximal choice such that the supersolution condition $v(i-1)-2v(i)+\bar{v}(i+1) \le f(i,v(i))$ is satisfied for $ i=7 $.
However, we check all possible choices below $ \bar{v}(8) $ (which automatically satisfy the supersolution condition at $i=7$), and choose $ v(8) $ such that the next provisional increment $ D_{9} $ satisfying the supersolution condition at $i=8$ is maximal. 
} \label{fig:discrete_path}
\end{center}
\end{figure}
%

\begin{remark}
  In the special case $\P \big(Z_0=1\big)=p$, $\P\big(Z_0=-1\big)=1-p$ the condition in the theorem holds iff
  $p > \frac 12 \big( 3-\sqrt{5} \big) \approx .38$.
\end{remark}


\begin{cor}
If, for some $F \in \Z$, 
$$ 
\E  \big(Z_0 \vee (-1+Z_1) \vee (-2+Z_2)\vee ...\big)>F,
$$
then, almost surely, there exists a function  $v :\Z \to \N_0$ such that $\Delta_{1} v(i) \le f(i,v(i)) - F $.
\end{cor}

\section{Continuum model}

\subsection{Setting}

In the continuum setting of equation \eqref{eq:evolution}, we consider $f\colon \R^2\times\Omega \to \R$ to be of the form
\begin{equation} \label{eq:cont_obst}
f(x,y,\omega) = \sum_{j\in\N} f_j(\omega)s(\rho)\varphi\left(\frac{x-x_j(\omega)}{\rho},\frac{y-y_j(\omega)}{\rho} \right),
\end{equation}
where $(x_j,y_j)$ is a 2-dimensional Poisson point process prescribing the centers of the obstacles. The random variables $f_j(\omega) \in [-\infty,\infty)$, which are assumed to be identically distributed and independent of the obstacle centers, prescribe the strength of each individual obstacle and must satisfy $\P \{f_1 \ge k\}>0$ for some $0<k\le 1$. 

In order to prove our result, we have to introduce a small parameter $\rho>0$ which determines the spatial extent of the obstacles, and will be chosen according to the statistical parameters of the obstacle distribution. Therefore, we assume that  the function $\varphi\in C_c^\infty(\R^2)$ satisfies $\varphi|_{[-1,1]^2}\ge 1$ and $\varphi|_{\R^2\setminus B_\alpha(0)} = 0$ for some $\alpha > \sqrt{2}$. As one can read off \eqref{eq:cont_obst}, this implies that the obstacles achieve their full strength $f_j$ on a square of side-length $2\rho$ and vanish outside a ball of radius $\alpha\rho$. Since small obstacles have a small effect on the propagating interface, we rescale their force by $s(\rho)$ such that their effect remains constant when changing $\rho$. It will turn out that the choice $s(\rho) = \frac{2}{\rho}$ is suitable.

Such an assumption of very small obstacles is also made in \cite{ForemanMakin}, where point obstacles are considered in a model for dislocation evolution. Their assumptions can be interpreted as the $\rho\to0$ limit of our model.

Again, we will show a.s.~existence of a non-negative viscosity supersolution of
\begin{equation} 
\label{eq:supersolution}
v''(x) - f( x , v(x) ) + F \le 0.
\end{equation}
The function $v$, that we construct, will be piecewise quadratic, 
and in points of non-differentiability the condition on viscosity solution will be trivially met, as from our construction it will hold that
\[ \lim_{x \nearrow a} v'(x) \ge \lim_{x \searrow a} v'(x). \]

To simplify this construction, we will work with the following setting. We split the obstacles into those with $f_j\ge k$ (and refer to them as \emph{positive}) and into those with $f_j<0$ (\emph{negative}). Obstacles for which $f_j\in[0,k)$ will be ignored. The centers of the positive/negative obstacles $ ( x_{j}^{\pm} , y_{j}^{\pm} ) $ are now distributed according to independent Poisson point processes on $\R^2$ with parameters $ \lambda^{\pm}$. 
Then all positive obstacles have full strength of at least $k s(\rho)$ in the squares of side-length $ 2 \rho $ centered at $ ( x_{j}^{+} , y_{j}^{+} ) $. We refer to these squares as \emph{cores} of the obstacles and denote $ks(\rho) = \frac{2k}{\rho} = S$.

The response of the medium to a given interface $(x,v(x))\subset \R^2$ may thus be estimated by
\begin{align*} f(x,v(x)) &\ge f_\rho( x , v(x) )  \\
&= \begin{cases} 
{\displaystyle \sum_{j \in \N} S \varphi( \tfrac{ x - x_{j}^{+} }{ \rho } \tfrac{ v(x) - y_{j}^{+} }{\rho} ) },
& \text{if for all  $i \in \N$ we have $\dist( (x,v(x)) , ( x_{i}^{-} , y_{i}^{-} ) ) > \alpha \rho$}, \\
-\i & \text{else},
\end{cases}
\end{align*}
dropping the dependence on $\omega\in\Omega$ for notational convenience.

In fact, as shown in Figure~\ref{figure:supersolution}, 
we will construct a supersolution that will completely avoid negative obstacles. 
Therefore, their precise strength does not matter. 
As for the positive obstacles, we will only employ that they exert force (at least) $S$ in the square of size $ 2 \rho $.

\begin{figure}[ht]
\begin{center}
\includegraphics[width=.7\textwidth]{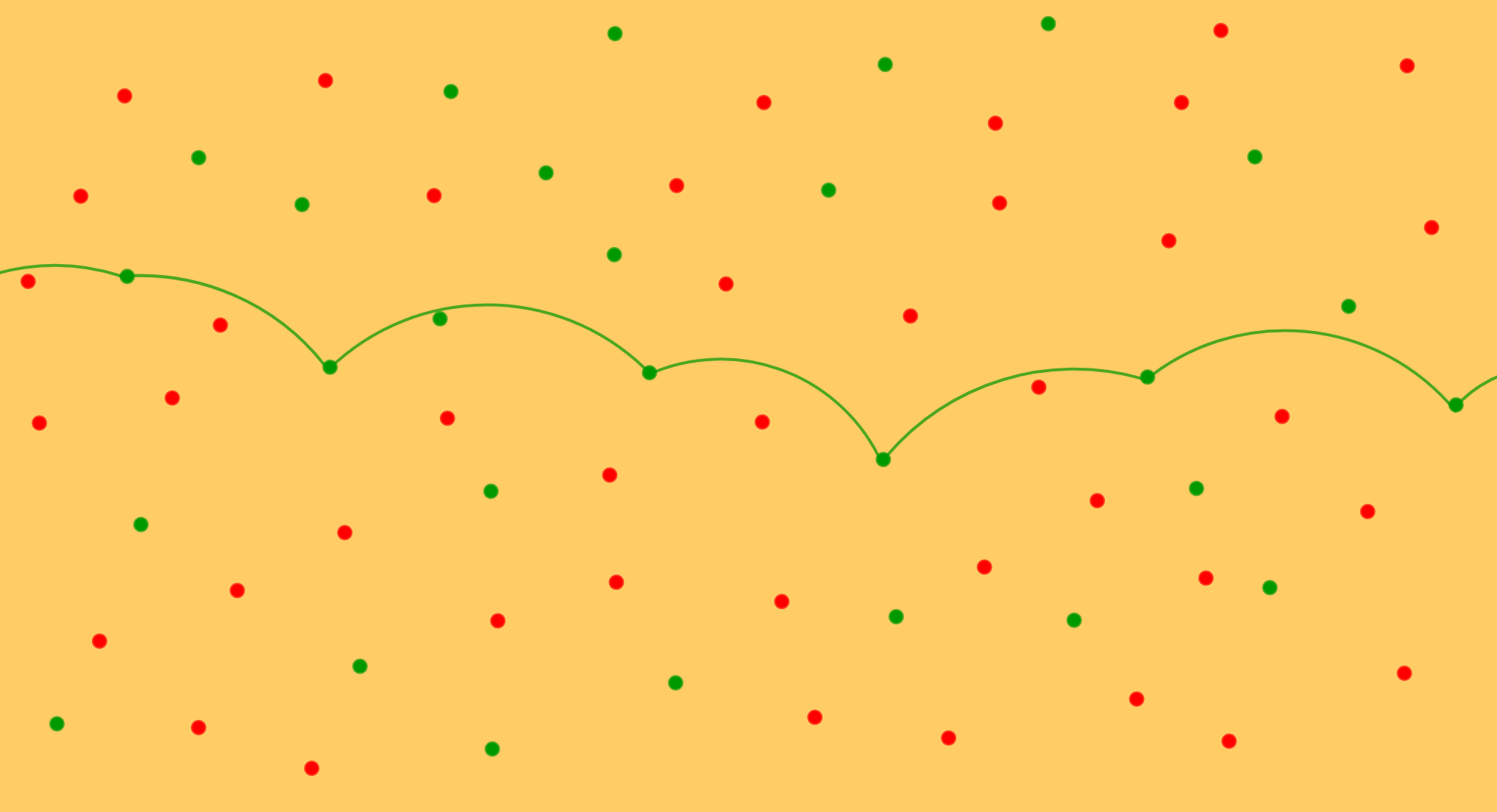}
\end{center}
\vspace{-5mm}
\caption{Idea for the construction. It follows from \eqref{eq:supersolution} that outside the obstacles any supersolution is concave, and inside the positive obstacles it may be convex. Since we are looking for a positive supersolution (green line), it must pass through sufficiently many positive obstacles (green dots), in which it can turn upwards, and avoid the negative ones (red dots).}
\label{figure:supersolution}
\end{figure}

\subsection{Localization}
First we have to localize enough positive obstacles to construct a blocking supersolution.
Therefore, let us define
\begin{equation}
\label{eq:def-Q}
Q_{i,j} := \Big( [ - \tfrac{l}{2} , \tfrac{l}{2} ] + i (l+d) \Big) \times [ j h, (j+1) h],
\quad i \in \Z, \ j \in \N. 
\end{equation}
We will consider the obstacles with entire cores lying in one of $ Q_{i,j} $.
As depicted in Figure~\ref{figure:decomposition},
we thus have columns of rectangles with length $l$ and height $h$, and between them there is a free space of width $d$.
\begin{figure}[ht]
\begin{center}
\includegraphics[width=.4\textwidth]{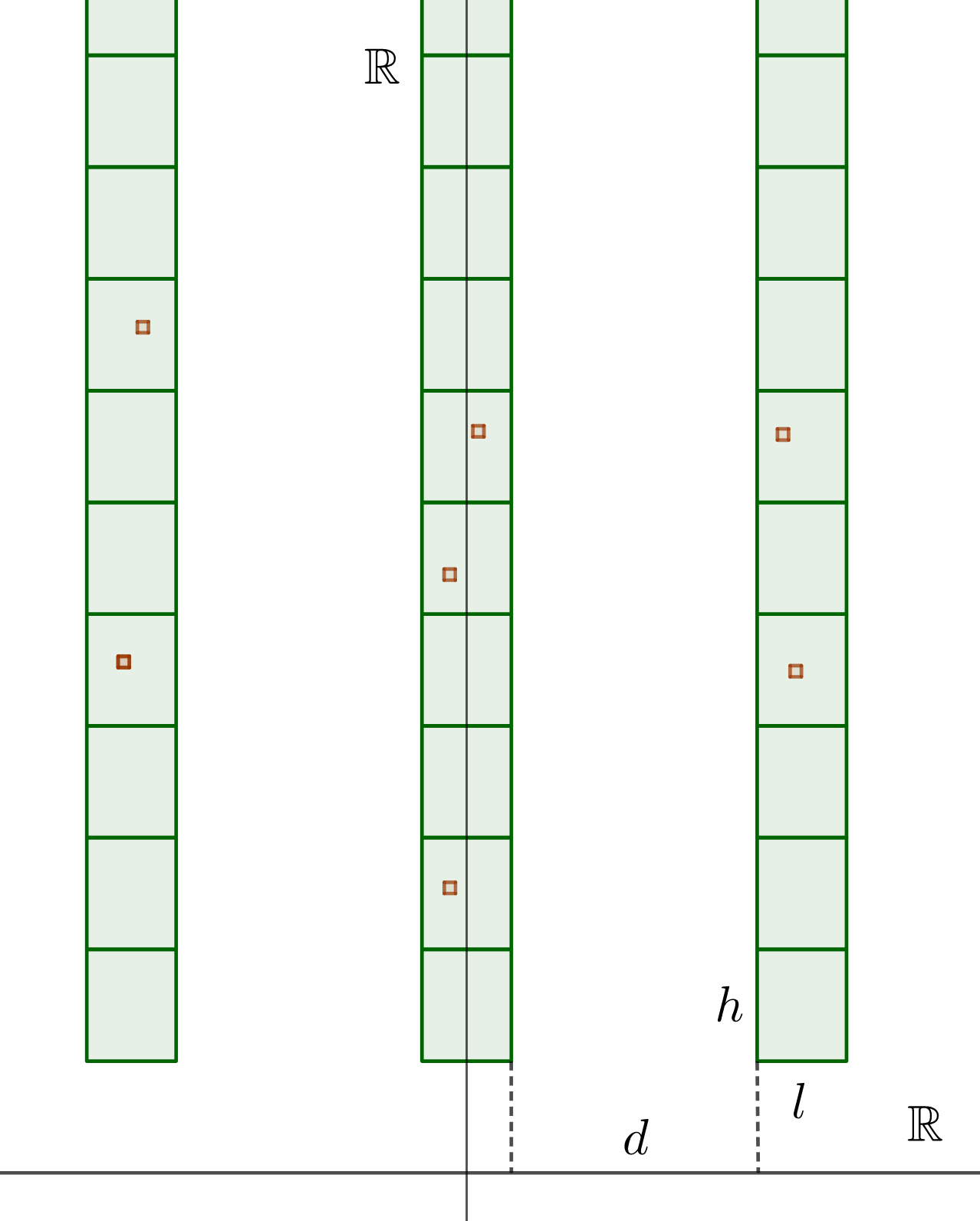}
\end{center}
\vspace{-5mm}
\caption{Decomposition of the upper half-plane. The $ Q_{i,j} $ from~\eqref{eq:def-Q} are the green rectangles of length $l$ and height $h$.}
\label{figure:decomposition}
\end{figure}
For now, these scales are still completely free.
%
%
We start at the height $h$ so that there is no intersection of the localized positive obstacles with the $x$-axis
as long as $ h > \alpha \rho - \rho $.
%

\subsection{Inside a core}
\label{subsect:inside-a-core}
\begin{figure}[ht]
\begin{center}
\includegraphics[width=0.2\textwidth]{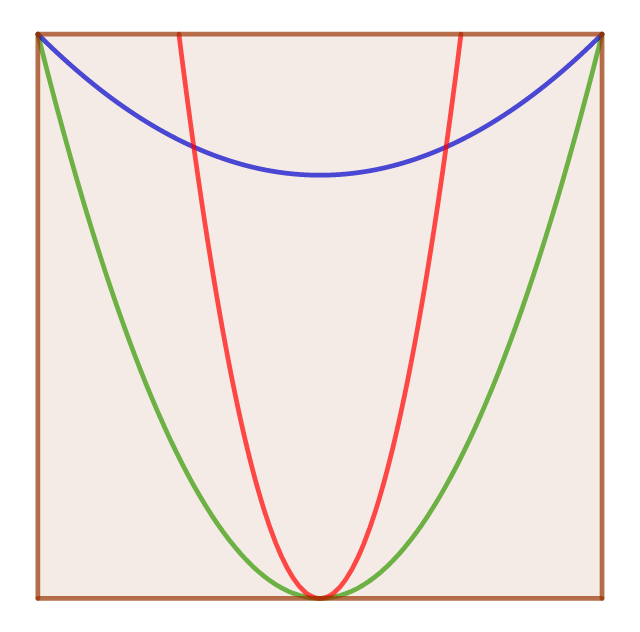}
\end{center}
\vspace{-5mm}
\caption{Parabolas within a core. We take parabolas through the upper corners of a core with the second derivative $ \frac{S}{2} $. For $ k < 4 $ the whole parabola lies within the core, the cases $k=1,4$ being depicted in blue and green. For larger $k$, the force of the obstacle could not be fully exploited, as suggested by the red line.}
\label{figure:ovira}
\end{figure}
%
Within the core of a positive obstacle, a sufficient condition to fulfil~\eqref{eq:supersolution} reads $ v''(x) \le S - F. $
Let us suppose $ F \le \frac{S}{2} $ (since in fact, as we will see, we may only pin $ F \ll S $).
We take the parabola with $ v'' = \frac{S}{2} $ that has its vertex on the mid vertical line of the core
and goes through the upper corners of the core.
Its inclination at the upper corners has modulus $ k $, and for $ k < 4 $ its vertex lies within the core, see Figure~\ref{figure:ovira}.
\subsection{An upper bound on the pinned force}
Let us first determine the force $F$ that we may block when we have two positive obstacles at a given distance if we suppose that there is no negative obstacle in the vertical strip between them.
The exiting points are, as specified in the previous section, the right resp.\ left corner. 
Let their positions for the sake of simplicity be $ A=(0,0) $ and $ B=(m,n) $ with $ m > 0 $.

\begin{figure}[ht]
\begin{center}
\includegraphics[width=.5\textwidth]{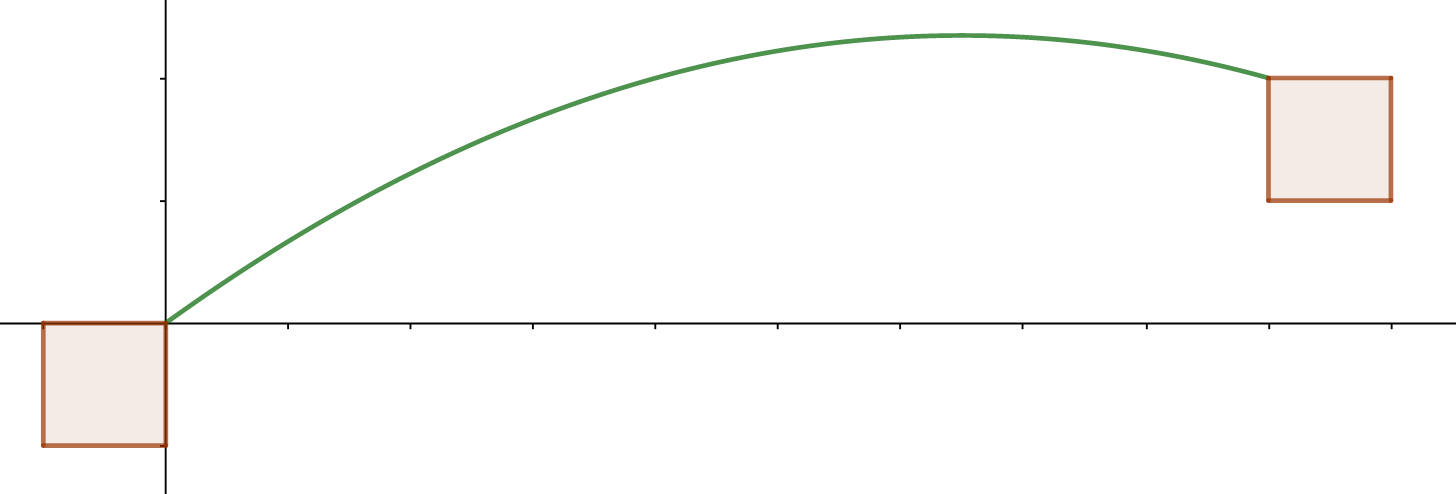}
\end{center}
\vspace{-5mm}
\caption{Parabola connecting two obstacles. The conditions are that it must connect the respective corners, fulfil the condition on the second derivative, and have the appropriate inclinations at the corners so as to be a building block for a viscosity supersolution.}
\label{figure:parabola}
\end{figure}

The conditions read
\begin{equation} \label{eq:viscosity}
v''(x) + F \le 0, \quad v'(0) \le k, \quad v'(m) \ge -k
\end{equation}
%
with the latter two ensuring that this function together with the parabolas in the cores forms a supersolution in the viscosity sense. Hence, necessary conditions for existence of a supersolution through the points $A$ and $B$ are
\begin{equation*}
k \ge \frac{F}{2} m + \frac{n}{m}
\quad \mbox{and} \quad
k \ge \frac{F}{2} m - \frac{n}{m}. 
\end{equation*} 
In the other direction, let $ k \ge \frac{F}{2} m + \frac{n}{m} $ and $ k \ge \frac{F}{2} m - \frac{n}{m} $. 
Then the parabola 
\[ v(x) = \left( \frac{F}{2} m + \frac{n}{m} \right) x - \frac{F}{2} x^{2} \]
is a supersolution between $A$ and $B$ since $ v'' + F = 0 $ and
\[ v(0) = 0, \quad
v(m) = n, \quad
v'(0) = \frac{F}{2} m + \frac{n}{m} \le k 
\quad \mbox{and} \quad
- v'(m) = \frac{F}{2} m - \frac{n}{m} \le k. \]
\begin{con}
Having two obstacles that produce an exiting inclination $k$ with horizontal distance $m$ and vertical $n$, we may pin any force of magnitude
\[ F \le 2 \frac{ k m - |n| }{ m^{2} }. \]
\end{con}
To control both distances, we will consider only parabolas between obstacles from the neighbouring columns
whose heights differs at most by one unit.
If we thus connect two obstacles in boxes $ Q_{i,j} $ and $ Q_{i+1,j+e} $ with $ e \in \{-1,0,1\} $, then $ d \le m < 2l + d $ and $ |n| < 2h $. 
We get a positive force if 
\begin{equation} \label{eq:1}
k d > 2 h.
\end{equation}
We then block at least
\begin{equation} \label{eq:2}
F \le 2 \frac{ k d - 2 h }{ (d+2l)^{2} }.
\end{equation}

\subsection{Length of the line between parabolas}

\begin{figure}[ht]
\begin{center}
\includegraphics[width=.59\textwidth]{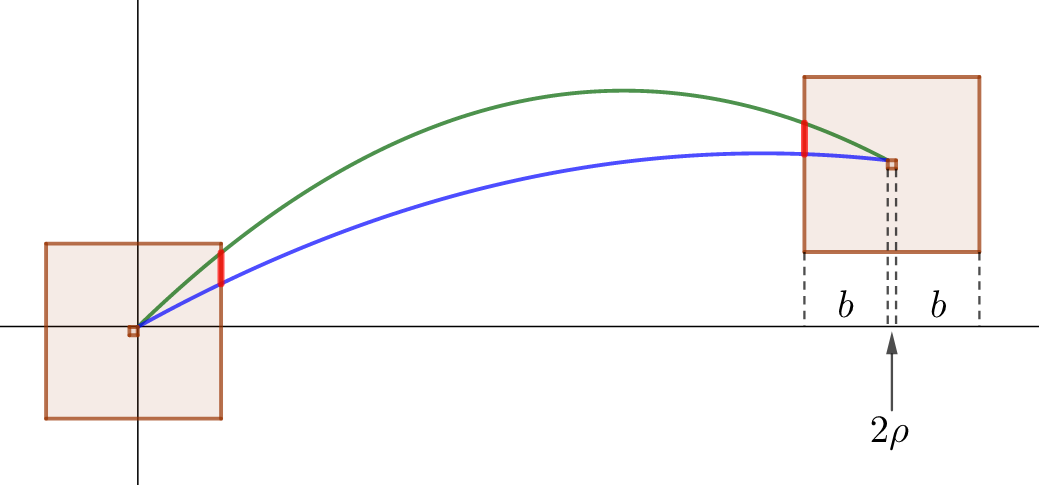}
\end{center}
\vspace{-5mm}
\caption{Length of the line between parabolas. Around each core (the tiny squares) we have a concentric square of thickness $b$ where no negative obstacle may lie. The negative obstacles thus intersect the largest number of parabolas if their centers lie on the intersection of the border of these squares with the ray of admissible parabola (the red lines).}
\label{figure:parabola-exit}
\end{figure}



We must take care not to intersect any negative obstacle with our supersolution. 
This we achieve by considering only the positive obstacles that fulfill the following two conditions:
\begin{itemize}
\item
There are no negative obstacles in the square centered at the center of the obstacle with the side $ 2 ( \rho + b ) $ where  $b$ is a new scale. Hence, we have a strip of thickness $b$ around the core as shown in Figure~\ref{figure:parabola-exit}.
\item
There are less than $N$ (for now arbitrary) negative obstacles in the whole region where parabolas may lie.
\end{itemize}
We consider a ray of possible parabolas and assess an admissible size of (negative) obstacles as follows:

The most problematic position for negative obstacles is on the border of the concentric square (the red line in Figure~\ref{figure:parabola-exit}).
Suppose all the negative obstacles lie on this line. Then they cover a height up to $ 2 (N-1) \alpha \rho $. We must, however, take into account that parabolas have a certain inclination. 
If a parabola meets this line outside the balls around the centers of the negative obstacles with radius $ 2 \alpha \rho $ (the dotted line in Figure~\ref{figure:blocking-obstacles}), then it does not intersect any negative obstacle (lying inside a ball with radius $ \alpha \rho $, the red circle in Figure~\ref{figure:blocking-obstacles}), as its inclination surely does not exceed 1. Thus, the negative obstacles block at most $ 4 (N-1) \alpha \rho $.
If this length is shorter than the half of the red line, we surely find a parabola around them on both sides.


\begin{figure}[ht]
\begin{center}
\includegraphics[width=.2\textwidth]{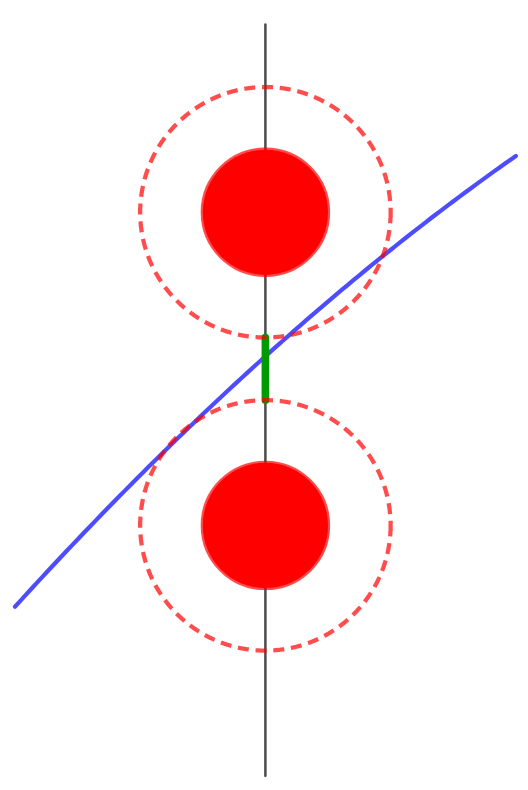}
\end{center}
\vspace{-5mm}
\caption{Estimating the length of the blocked line. If a negative obstacle lies on the critical line, it may happen that it intersects a parabola outside of this line, as the parabola has a certain inclination. However, since for sure its inclination is less than 1, this cannot happen if we allow parabolas to pass this line only on the parts that are at least $ 2 \alpha \rho $ away from the center of any negative obstacle.}
\label{figure:blocking-obstacles}
\end{figure}

If a ray of parabolas is given by $ u_{1} $ and $ u_{2} $ with
\[ u_{1}''(x) = - F_{1}
\quad \mbox{and} \quad
u_{2}''(x) = - F_{2}, \]
where $ F_{1} > F_{2} $ 
then the length of the part of the border of each square determined by this ray is given by 
\[ u_{1}( b ) - u_{2}( b ) 
= u_{1}( m - b ) - u_{2}( m - b ) 
= \frac{ F_{1} - F_{2} }{2} b ( m - b ). \]
We also notice: 
Since we restrict ourselves to $ k \le 1 $, that parabolas surely exit the concentric square at the right resp.\ left side.
%
\subsection{Position of vertex}
The parabola 
\[ u(x) = \left( \frac{F}{2} m + \frac{n}{m} \right) x - \frac{F}{2} x^{2} \]
has its vertex in
\[ x_{0} = \frac{ \frac{F}{2} m + \frac{n}{m} }{ F }, \quad
y_{0} = \frac{ \left( \frac{F}{2} m + \frac{n}{m} \right)^{2} }{ 2 F }. \]
Since we want to control its height only between $ (0,0) $ and $ (m,n) $, we have to look at the case $ 0 \le x_{0} \le m $. 
In this case
\begin{equation} \label{eq:3}
0 \le \frac{ \frac{F}{2} m + \frac{n}{m} }{ F } \le m
\quad \mbox{or} \quad
F m^{2} \ge 2 |n|.
\end{equation}  
It suffices to consider $ n \ge 0 $.
If $F$ is too small to fulfil \eqref{eq:3}, 
the whole piece of the parabola lies lower than the higher obstacle.
Otherwise, let us allow only parabolas with $ y_{0} \le 2h $.
Since
$ y_{0} = \frac{F}{2} x_{0}^{2} \le \frac{F}{2} m^{2}, $
a sufficient condition for this is
\begin{equation} \label{eq:4}
F \le \frac{ 4 h }{ (2l+d)^{2} }.
\end{equation}
For every such $F$, no point of the corresponding parabola between the obstacles lies more than $ 2h $ higher than the lower obstacle. 
Thus, these parabolas surely lie in rectangles of the height $3h$ depicted in Figure~\ref{figure:bound-parabola}. 

\begin{figure}[ht]
\begin{center}
\includegraphics[width=.4\textwidth]{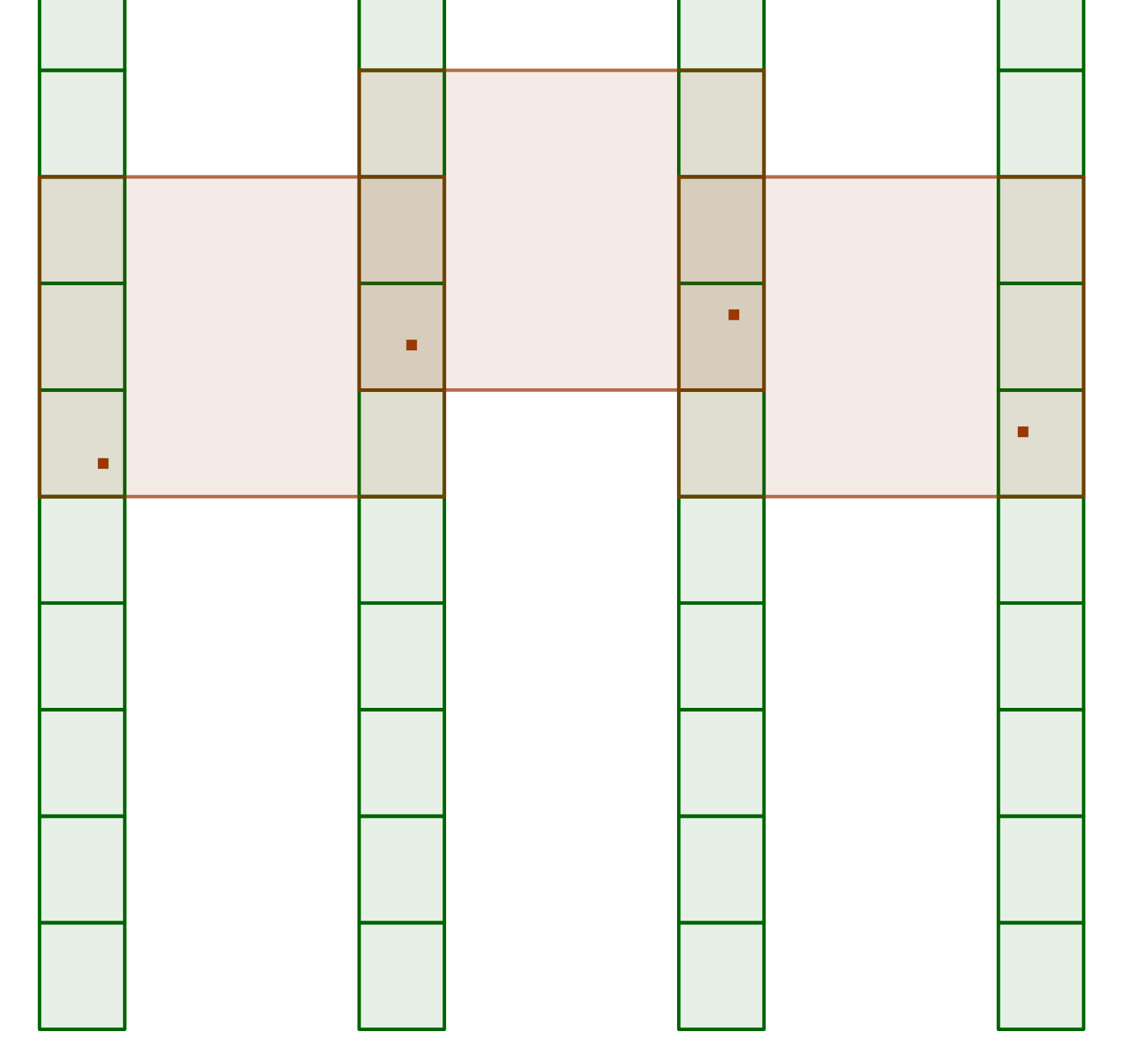}
\end{center}
\vspace{-5mm}
\caption{Admissible boxes for parabolas. We set the condition on the connecting parabola that its highest point lies at most $ 2 h $ higher that the lower obstacle. For the depicted situation, the respective parabolas lie in the shaded rectangles with side-lengths $ 2l+d $ and $3h$. }
\label{figure:bound-parabola}
\end{figure}

%
%
%
%

\subsection{Percolation}
\label{subsect:percolation}
Now we must obtain a.s.~a sequence of ``good'' rectangles $ \{ Q_{i, j(i) } \}_{i \in \Z} $ that contain positive obstacles and avoid negative ones
such that $ | j(i+1) - j(i) | \le 1 $ for all $ i \in \Z $.
This can be formulated as a problem in Lipschitz percolation as follows.

Let $d \ge 1$ and consider {\em $d$-independent} site percolation on $\Z^{n+1}$ with parameter $p \in [0,1]$, i.e.~$\xi(u)$, $u\in \Z^{n+1}$
are random variables taking values in $\{0,1\}$ such that $\P\big(\xi(u)=1\big) =p$ for every $u$ and  $\xi(u)$, $u \in M$ are independent
whenever $M$ is a subset of $\Z^{n+1}$ for which
the following holds: if $u\neq v$ are elements of $M$, then either $|u_{n+1}-v_{n+1}|\ge d$ or $(u_1,...,u_n)\neq (v_1,...,v_n)$.
Note that the case $d=1$ corresponds to the independent site percolation set-up. If $\xi(u)=1$, then we say that the site $u$ is {\em open} and
{\em closed} otherwise. The last component $u_{n+1}$ of $u \in \Z^{n+1}$ plays a different role compared to $(u_1,...,u_n)$. We will
sometimes call the last coordinate \emph{vertical} and the others \emph{horizontal}.

In this set-up the following result holds.

\begin{prop}
\label{prop:percolation}
There is a critical probability $ p_{0}=p_0(n,d) \in (0,1) $ such that for every $ p > p_{0} $ there is a Lipschitz percolation cluster,
i.e.~almost surely, there exists an {\em open Lipschitz surface}, i.e.~a function $\phi:\Z^n \to \N$ such that $|\phi(z)-\phi(\bar z)|\le 1$ whenver $|z-\bar z|=1$
and that $\xi(\phi(z))=1$ for all $z \in \Z^n$.
\end{prop}
\begin{proof}
For $d=1$ this is \cite[Theorem 1]{DDGHS} or \cite[Theorem 1]{GH}. In the general case, we will sketch the proof following that of \cite[Theorem 1]{GH} rather closely.

For each $ u \in \Z^{n+1} $, $h(u) := u_{n+1}$ is its \emph{height}. Let 
$ e_{1} , \ldots , e_{n} , e_{n+1} $ be the standard basis vectors in $\R^{n+1}$.

A $ \lambda $-\emph{path} from $u$ to $v$ is a finite sequence of distinct states 
$ x_{0} , \ldots , x_{k} $ with $ x_{0} = u $ and $ x_{k} = v $ such that for every $ i = 1 , \ldots , k $
we have
\begin{itemize}
\item either $ x_{i} - x_{i-1} = e_{n+1} $, in which case we speak of a \emph{step upwards}, 
\item or $ x_{i} - x_{i-1} \in \{ \pm e_{1} - e_{n+1} , \ldots , \pm e_{n} - e_{n+1} \} $ 
that is a \emph{step downwards}.
\end{itemize}
A $ \lambda $-path is called \emph{admissible} if the endpoint of every step upwards is closed. We denote by $ u \rightarrowtail v $ the event that there exists
an admissible $ \lambda $-path from $u$ to $v$.
As was shown in \cite[Section 4]{GH}, an open Lipschitz surface exists if and only if there exists some $h_0>0$ such that there does not exist any $u=(z,0)$ with $z \in \Z^n$ and any
$h \ge h_0$ such that $ u \rightarrowtail (0,h)$ (for a slightly simpler proof [in a slightly more general set-up] avoiding the concept of {\em hills} and {\em mountains}
see \cite{DST,DSW}).
Therefore, we have to show that for sufficiently large $p<1$ and large $h \in \N$ it is unlikely that there exists an admissible path starting from any point of the form
$(z,0)$ to $(0,h)$. Due to a possible lack of translation invariance, we need to modify the proof slightly compared to those mentioned above.

For a $ \lambda $-path $\pi$ from $(z,0)$ to $(0,h)$, we denote by $U$ and $D$ the number of steps upwards resp.~downwards. Then $h=U-D$ and $|z|\le D$ and $\pi$ is admissible
if {\em every} step upwards ends at a closed site. 
Due to $d$-independence (and the fact that a $ \lambda $-path consists of distinct points) the probability of this happening is at most  
$ q^{ \lfloor \frac{U-1}{d} \rfloor + 1 } \le q^{ \frac{U}{d} } $. 
Figure~\ref{figure:percolation-dependent} shows such an admissible $ \lambda $-path. 
The sites in red must be closed, and we obtain the bound on the probability by realizing that at least $ \lfloor \frac{U-1}{d} \rfloor + 1 $ of them are independent. In every column, we may simply take the highest one and then every $d$-th one that lies in the path. 
This bound is also optimal, take, e.g., the straight vertical path $ (0,0) \rightarrowtail (0, U d + 1 ) $.

\begin{figure}[ht]
\begin{center}
\includegraphics[width=0.5\textwidth]{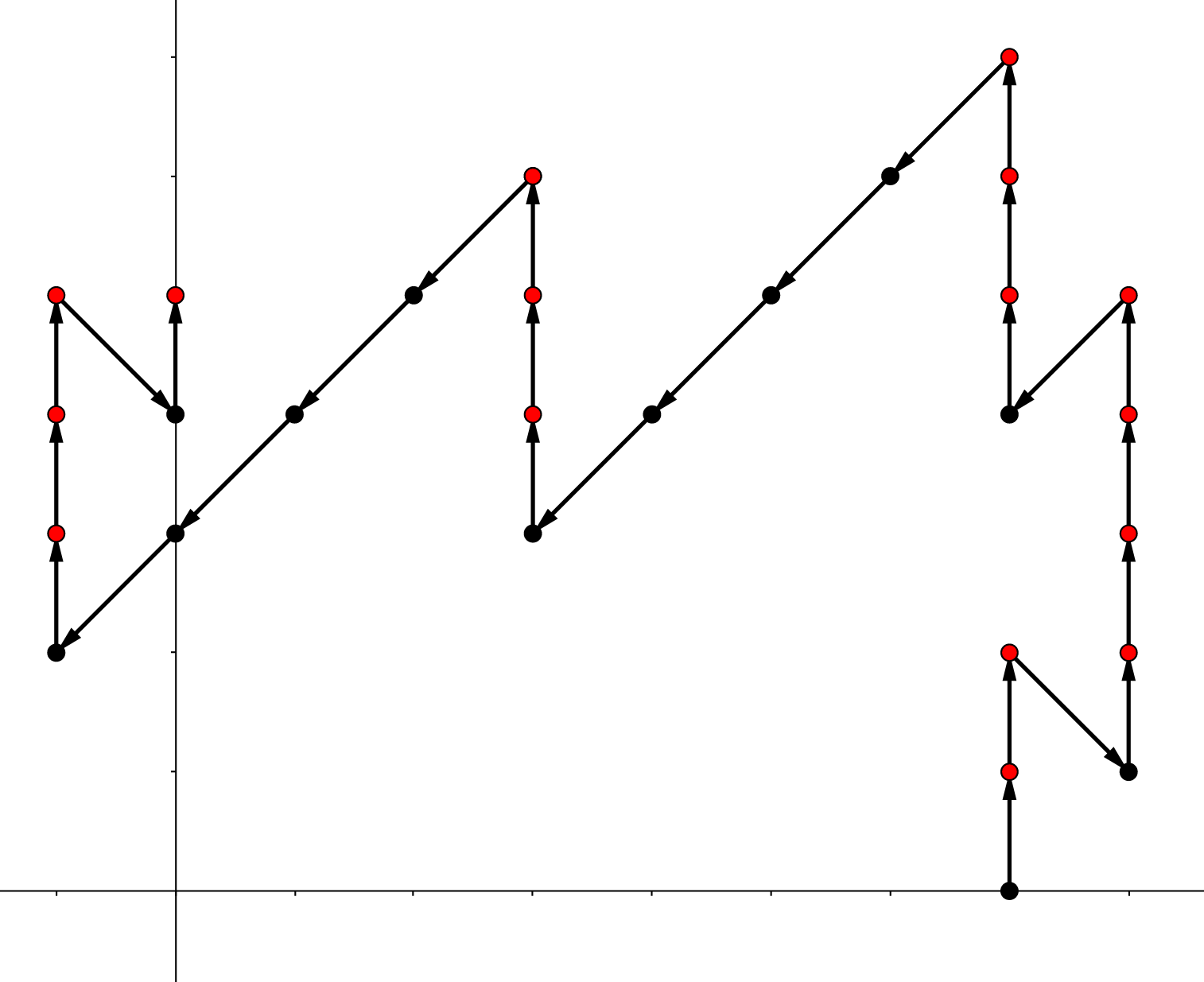}
\end{center}
\vspace{-5mm}
\caption{Admissible $\lambda$-path. A $\lambda$-path between two points consists of steps upwards and diagonally downwards.
It is admissible if all steps upwards end in a closed site. Here the red-coloured sites must be closed while the states of the others are arbitrary.}
\label{figure:percolation-dependent}
\end{figure}

The rest of the proof is as usual (just replacing $q^U$ by $q^{ \frac{U}{d} }$):
the expected number of admissible $ \lambda $-paths from  $(z,0)$ to $(0,h)$
is upper bounded by 
\begin{align*}
  \sum_{ U-D = h , D \ge |z| } &\binom{U+D}{U}(2n)^D q^{ \frac{U}{d} } \le  \sum_{ U-D = h , D \ge |z| } 2^{U+D}(2n)^D q^{ \frac{U}{d} }\\
  &= \sum_{D \ge |z| }2^{2D+h}(2n)^D q^{ \frac{D+h}{d} }
  = 2^hq^{h/d}(8nq^{1/d})^{|z|}\frac 1{1-8nq^{1/d}},
\end{align*}
provided that $8nq^{1/d}<1$. In this case the last expression is summable over all $h >h_0$ and all $z \in \Z^n$ and the double sum converges to 0 as $h_0 \to \infty$,
so the claim of the proposition follows (with $p_0=1-\Big(\frac 1{8n}\Big)^d$).
\end{proof}

\subsection{Putting it together}
\label{subsect:putting-it-together}
Now let us state our problem in terms of Lipschitz percolation.
Our sites $ (i,j) \in \Z^{2} $ are boxes $ Q_{i,j} $. 
We impose three conditions on a box to be open:
\begin{itemize}
\item 
$ Q_{i,j} $ must contain the core of a positive obstacle.
Thus we get a condition on the heights and lengths. 
That means that $ 2 \rho < l,h $ and that the probability is
\[ \P(\mbox{$ Q_{i,j} $ contains the core of a positive obstacle}) = 1 - e^{ - \lambda^{+} ( l - 2 \rho )( h - 2 \rho ) }. \]
\item
Moreover, a core together with a strip of width $b$ around it must not intersect any negative obstacle. 
We need $ b < \frac{d}{2} $ in order for probabilities to be independent in the horizontal direction and, e.g., $ b < h $ to have limited dependence (more precisely, 2-independence) in the vertical direction. 
The probability then reads
\[ \P(\mbox{strip intersects no negative obstacles}) \ge e^{ - \lambda^{-} ( 2 b + 2 \rho + 2 \alpha \rho )^{2} }. \]
\item 
Lastly, we want the rectangle ``around'' a positive obstacle with length $ l + d $ and height $ 6 h $ (as in Figure~\ref{figure:percolation}) to contain less than $N$ centers of negative obstacles.
If we denote $ V := 6h(l+d) $, then
\begin{multline*}
\P( \mbox{rectangle intersects less than $N$ centers of negative obstacles} ) = \\
= e^{ - \lambda^{-} V } \sum_{k=0}^{N-1} \frac{ ( \lambda^{-} V )^{k} }{ k! } 
\ge e^{ - \lambda^{-} V } 
\left( e^{ \lambda^{-} V } - e^{ \lambda^{-} V } \frac{ ( \lambda^{-} V )^{N} }{ N! } \right)
= 1 - \frac{ ( \lambda^{-} V )^{N} }{ N! }
\end{multline*} 
%
%
(If we want the last two events to independent, we may exclude here the square with side $ 2 ( b + \rho ) $. The inequality, however, still holds.)
\end{itemize}

\begin{figure}[ht]
\begin{center}
\includegraphics[width=.5\textwidth]{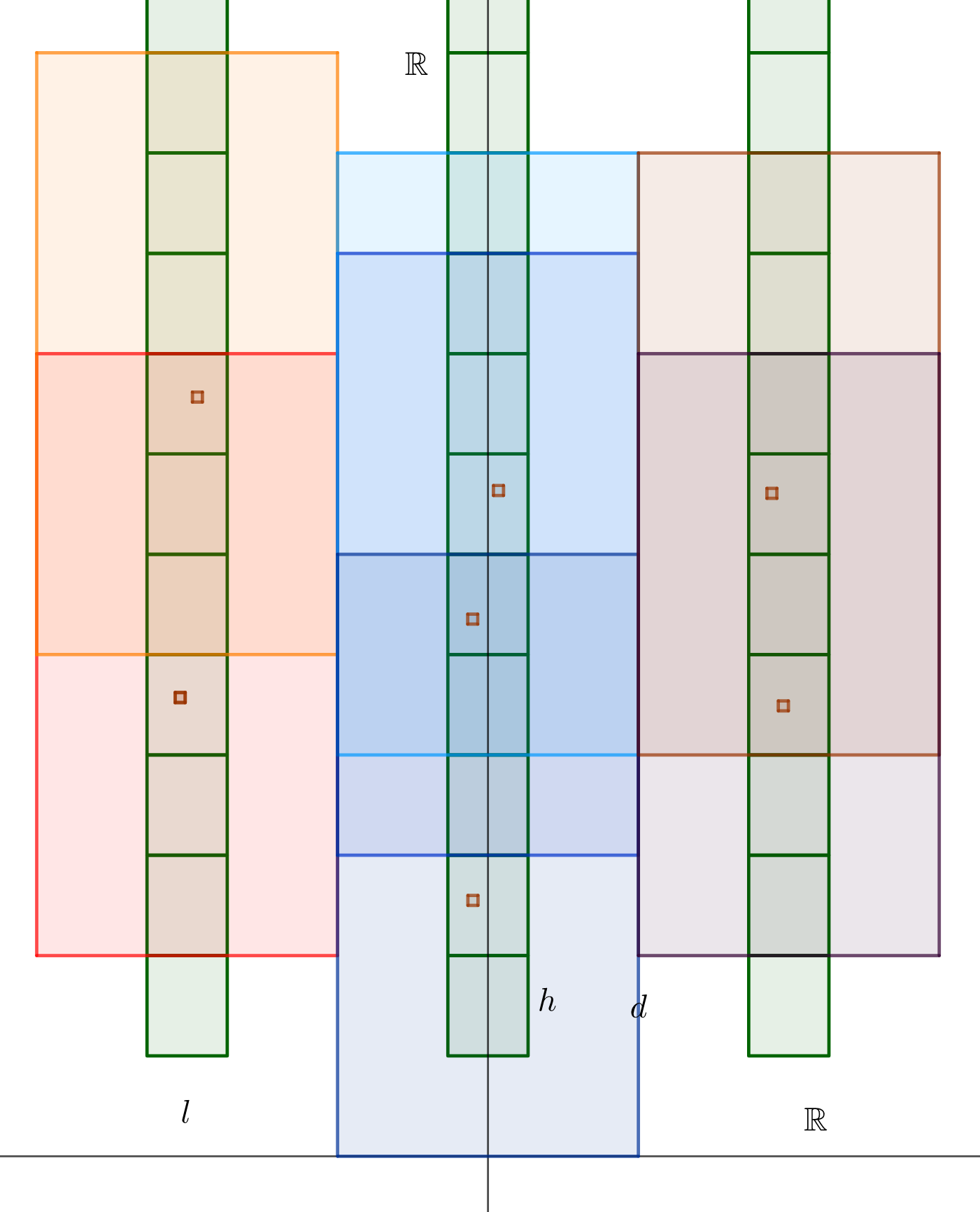}
\end{center}
\vspace{-5mm}
\caption{Boxes corresponding to percolation sites. A site $ (i,j) $ is open if $ Q_{i,j} $ contains the center of a positive obstacle, if the square of thickness $b$ around the core of this obstacle contains no center of negative obstacles, and if the larger rectangle of side-lengths $l+d$ and $6h$ contains less than $N$ centers of negative obstacles.
In this figure some of these larger rectangles are depicted in different colours, e.g., the red one belonging to $ (-1,4) $.}
\label{figure:percolation}
\end{figure}

Hence,
\[ \P(\mbox{$ Q_{i,j} $ is open}) \ge 
( 1 - e^{ - \lambda^{+} ( l - 2 \rho )( h - 2 \rho ) } ) e^{ - 4 \lambda^{-} ( b + ( 1 + \alpha ) \rho )^{2} }
\left( 1 - \frac{ ( 6 \lambda^{-} h ( l + d ) )^{N} }{ N! } \right). \]
We may surely employ Proposition~\ref{prop:percolation} if the right-hand side is bigger than $ p_{0} = p_{0}(1,6) $
since the sites are 6-independent.
The scales can be chosen in the following way:
\begin{enumerate}
\item
Let us suppose $ l , h \ge 4 \rho $.
Set $ d := l $ and $ h := \frac{k d}{4} $ (to obey \eqref{eq:1}), and choose $l$ so large that
\begin{equation}
\label{eq:a}
1 - e^{ - \lambda^{+} ( l - 2 \rho )( h - 2 \rho ) }
\ge 1 - e^{ - \lambda^{+} h l / 4 }  \ge p_{0}^{1/3}. 
\end{equation}
\item
Suppose $ ( 1 + \alpha ) \rho \le b. $
We choose $b$ small enough so that $ b < h $ (thus also $ b < \frac{d}{2} $) and 
\begin{equation}
\label{eq:b}
e^{ - 4 \lambda^{-} (2b)^{2} } \ge p_{0}^{1/3}.
\end{equation}
\item
Finally, choose $ N \in \N $ so that 
\begin{equation}
\label{eq:c}
 1 - \frac{ ( 6 \lambda^{-} h ( l + d ) )^{N} }{ N! } > p_{0}^{1/3}. 
\end{equation}
\end{enumerate}
The percolation result is now applicable, and we get a Lipschitz function between open sites.
Each of these open sites contains a positive obstacle,
and by \eqref{eq:2} and \eqref{eq:4}, 
we block by appropriately chosen parabola at least all
$ F \le \min \left\{ \frac{4h}{ (2l+d)^{2} } , 2 \frac{kd-2h}{ (2l+d)^{2} } \right\} $.
Making the choices as above, we arrive at
\begin{equation}
\label{eq:F}
F^{*} := \frac{1}{2} \min \left\{ \frac{4h}{ (2l+d)^{2} } , 2 \frac{kd-2h}{ (2l+d)^{2} } \right\} = \frac{k}{ 18 l }. 
\end{equation}
We choose the interval of admissible forces $ [ F^{*} , 2 F^{*} ] $. 
The corresponding parabolas cut on the side of the square a line of length
\[ \frac{ 2 F^{*} - F }{2} b (m-b) 
\ge \frac{F^{*}}{2} b (d-b) 
= \frac{ k b (l-b) }{ 36 l }. \]
Surely, if 
\begin{equation}
\label{eq:rho}
\frac{1}{2} \frac{ k b (l-b) }{ 36 l } \ge 4 N \alpha \rho,
\quad \mbox{and thus} \quad 
\rho \le \frac{ k b (l-b) }{ 288 \alpha N l },
\end{equation}
there is a parabola that does not intersect any negative obstacle
with center in the two rectangles of size $ (l+d) \times 6h $ belonging to the open sites. 
We see that the assumptions on $ \rho $ from (a) and (b) are automatically fulfilled.
Clearly, we also have $ \alpha \rho < h, d $, and thus this parabola also cannot intersect any negative obstacle with center outside the rectangles.

Thus, we can show
\begin{theo}
Let us suppose the following:
\begin{itemize}
\item 
Distribution: We have two independent Poisson point processes with parameters $ \lambda^{\pm} $ with $ ( x_{j}^{\pm} , y_{j}^{\pm} ) $, $ j \in \N $ being the corresponding positions of random points.
\item
Shape: A non-negative function $ \varphi \in C_{c}^{\i}( \R^{2} ) $ fulfils $ \varphi \ge 1 $ on $ [-1,1]^{2} $ and $ \supp \varphi \subset  $ for some $ \alpha > \sqrt{2} $.
\item
Strength: $ k \in (0,1] $.
\end{itemize}
Define for every $ \rho > 0 $:
\[ f_{\rho}(x,y, \omega ) 
:= \sum_{j} \frac{2k}{ \rho } \
\varphi \left( \frac{ x-x_{j}^{+}( \omega ) }{ \rho } , \frac{ y-y_{j}^{+}( \omega ) }{ \rho } \right). \]
Then there exist $ \rho^{*} > 0 $ and $ F^{*} > 0 $ such that a.s.~there exists a function $ v : \R \times \Omega \to (0,\i) $ that satisfies
\begin{equation}
\label{eq:condition}
v''(x) - f_{\rho^{*}}( x , v(x) ) + F^{*} \le 0
\end{equation} 
in the viscosity sense and $ d( ( x, v(x) ) , ( x_{j}^{-} , y_{j}^{-} ) ) > \alpha \rho^{*} $ for all $ j \in \N $.
\end{theo}
\begin{proof}
We choose the scales $ l , d , h , b , N $ as described in \eqref{eq:a}--\eqref{eq:c}. 
Thus we obtain Lipschitz percolating boxes $ \{ Q_{i,j(i)} \}_{ i \in \Z } $.
Then we define 
\[ F^{*} := \frac{k}{18 l} 
\quad \mbox{and} \quad 
\rho^{*} := \frac{ k b ( l - b ) }{ 144 \alpha N l } \]
according to \eqref{eq:F} and \eqref{eq:rho}.
Between two obstacles from adjacent boxes, for some $ F \in [ F^{*} , 2 F^{*}] $
we may a.s.~find a parabola $ v'' + F = 0 $ that does not intersect any negative obstacle.
($F$ need not be the same for different pairs.)

We define the supersolution $v$ in a piecewise manner. 
Between the cores from the Lipschitz percolating boxes, we take the parabolas from above.
Inside the cores, we take parabolas with $ v'' = \frac{S}{2} = \frac{k}{\rho} $ as described in Subsection~\ref{subsect:inside-a-core}.
We notice that the assumption $ 2 F^{*} \le \frac{k}{\rho} $ made there is fulfilled. 
Moreover, the cores also do not intersect any negative obstacle.
At the edges of core, $v$ may be non-differentiable. 
However, due to \eqref{eq:viscosity}, $v$ suffices the inequality \eqref{eq:condition} in these points in the viscosity sense.
\end{proof}
\begin{remark}
Clearly, also for every $ \rho \le \rho^{*} $ and $ F \le F^{*} $ we may a.s.~find such a function.
\end{remark}

\bibliographystyle{abbrv}
\bibliography{refs}

\begin{thebibliography}{10}

\bibitem{Bodineau:2013ur}
T.~Bodineau and A.~Teixeira.
\newblock {Interface Motion in Random Media}.
\newblock {\em Communications in Mathematical Physics}, 334(2):843--865, Mar.
  2015.

\bibitem{BrazovskiiNattermann}
S.~Brazovskii and T.~Nattermann.
\newblock Pinning and sliding of driven elastic systems: from domain walls to
  charge density waves.
\newblock {\em Advances in Physics}, 53(2):177--252, Mar. 2004.

\bibitem{Crandall:1992kn}
M.~G. Crandall, H.~Ishii, and P.-L. Lions.
\newblock {User's guide to viscosity solutions of second order partial
  differential equations}.
\newblock {\em American Mathematical Society. Bulletin. New Series},
  27(1):1--67, 1992.

\bibitem{DDGHS}
N.~Dirr, P.~W. Dondl, G.~R. Grimmett, A.~E. Holroyd, and M.~Scheutzow.
\newblock Lipschitz percolation.
\newblock {\em Electron. Commun. Probab.}, 15:14--21, 2010.

\bibitem{DDS}
N.~Dirr, P.~W. Dondl, and M.~Scheutzow.
\newblock Pinning of interfaces in random media.
\newblock {\em Interfaces Free Bound.}, 13(3):411--421, 2011.

\bibitem{DirrYip}
N.~Dirr and N.~K. Yip.
\newblock Pinning and de-pinning phenomena in front propagation in
  heterogeneous media.
\newblock {\em Interfaces Free Bound.}, 8(1):79--109, 2006.

\bibitem{DondlScheutzow:12}
P.~W. Dondl and M.~Scheutzow.
\newblock Positive speed of propagation in a semilinear parabolic interface
  model with unbounded random coefficients.
\newblock {\em Netw. Heterog. Media}, 7(1):137--150, 2012.

\bibitem{DondlScheutzow:17}
P.~W. Dondl and M.~Scheutzow.
\newblock Ballistic and sub-ballistic motion of interfaces in a field of random
  obstacles.
\newblock {\em Ann. Appl. Probab.}, 27(5):3189--3200, 2017.

\bibitem{DST}
P.~W. Dondl, M.~Scheutzow, and S.~Throm.
\newblock Pinning of interfaces in a random elastic medium and logarithmic
  lattice embeddings in percolation.
\newblock {\em Proc. Roy. Soc. Edinburgh Sect. A}, 145(3):481--512, 2015.

\bibitem{DSW}
A.~Drewitz, M.~Scheutzow, and M.~Wilke-Berenguer.
\newblock Asymptotics for {L}ipschitz percolation above tilted planes.
\newblock {\em Electron. J. Probab.}, 20:Paper No. 117, 23, 2015.

\bibitem{ForemanMakin}
A.~J.~E. Foreman and M.~J. Makin.
\newblock Dislocation movement through random arrays of obstacles.
\newblock {\em Philosophical Magazine}, 14(131):911--924, Oct. 1966.

\bibitem{GH}
G.~R. Grimmett and A.~E. Holroyd.
\newblock Geometry of {L}ipschitz percolation.
\newblock {\em Ann. Inst. Henri Poincar\'{e} Probab. Stat.}, 48(2):309--326,
  2012.

\bibitem{Nattermann:97}
T.~Nattermann.
\newblock Theory of the random field ising model.
\newblock In {\em Series on Directions in Condensed Matter Physics}, pages
  277--298. World Scientific, dec 1997.

\end{thebibliography}

\end{document}